\documentclass[11pt,reqno]{amsart}

\usepackage{amsfonts, amssymb, amscd}
\usepackage{amsmath}
\usepackage{verbatim}
\usepackage{amssymb}
\usepackage{mathrsfs}
\usepackage{graphicx}
\usepackage{cite}
\usepackage[all]{xy}
\usepackage{tikz}
\usepackage[toc,page]{appendix}
\usepackage{tikz-cd}

\usepackage{graphicx}
\usepackage{float}

\newtheorem{theorem}{Theorem}[section]
\newtheorem{lemma}[theorem]{Lemma}
\newtheorem{proposition}[theorem]{Proposition}
\newtheorem{definition}[theorem]{Definition}
\newtheorem{example}[theorem]{Example}
\newtheorem{corollary}[theorem]{Corollary}
\newtheorem{remark}[theorem]{Remark}
\newtheorem{conjecture}[theorem]{Conjecture}
\numberwithin{equation}{section}
\begin{document}
\title[On strong exceptional collections of maximal length]{On strong exceptional collections of line bundles of maximal length on Fano toric Deligne-Mumford stacks }

\begin{abstract}
We study strong exceptional collections of line bundles on Fano toric Deligne-Mumford stacks $\mathbb{P}_{\mathbf{\Sigma}}$ with rank of Picard group at most two. We prove that any strong exceptional collection of line bundles generates the derived category of $\mathbb{P}_{\mathbf{\Sigma}}$, as long as the number of elements in the collection equals the rank of the (Grothendieck) $K$-theory group of $\mathbb{P}_{\mathbf{\Sigma}}$.
\end{abstract}

\author{Lev Borisov}
\address{Department of Mathematics\\
Rutgers University\\
Piscataway, NJ 08854} \email{borisov@math.rutgers.edu}

\author{Chengxi Wang}
\address{Department of Mathematics\\
Rutgers University\\
Piscataway, NJ 08854} \email{cw674@math.rutgers.edu}

\maketitle

\tableofcontents
\section{Introduction}\label{section1}
Constructing phantom and quasiphantom subcategories of the derived category of coherent sheaves
on smooth projective varieties has attracted considerable interest over the years. A quasi-phantom subcategory is an admissible subcategory with
trivial Hochschild homology and with a finite Grothendieck group. A phantom subcategory is an admissible subcategory with trivial Hochschild
homology and a trivial Grothendieck group.

\smallskip

The authors of \cite{BGrS,AO,GS} construct some quasi-phantom subcategories as semiorthogonal complements
to exceptional collections of maximal possible length on certain surfaces of general type for which
$q = p_g = 0$. Moreover, the Grothendieck group of a quasiphantom is
isomorphic to the torsion part of the Picard group of a corresponding surface.
It is natural to ask whether there exists a phantom as a semiorthogonal complement
to an exceptional collection of maximal length on a simply connected surface of general type with
$q = p_g = 0$ like a Barlow surface. It was achieved by B$\ddot{\mathrm{o}}$hning, H-Ch. Graf von Bothmer, L. Katzarkov, and
P. Sosna in \cite{BGKS}. They show that in a small neighbourhood of the surface constructed by Barlow in the moduli space of determinantal Barlow surfaces, the generic surface has a semiorthogonal decomposition of its derived category into a length 11 exceptional sequence of line bundles and a category with trivial Grothendieck group and Hochschild homology.

\smallskip

Moreover, in \cite{SOrlov}, S. Gorchinskiy and D. Orlov construct geometric phantom categories by
considering admissible subcategories generated by the tensor product of two quasi-phantoms for which
orders of their (Grothendieck) $K$-theory groups are coprime. They also prove that these phantom
categories have trivial $K$-motives and, hence, all
their higher $K$-groups are trivial too.
\cite{PSo} Under certain
assumptions on the semi-orthogonal decomposition, this result has implications for the structure of the Chow motive of a variety admitting a
phantom category.

\smallskip

However, \cite{SNo} shows that there are no quasi-phantoms, phantoms or universal phantoms in the derived category of smooth projective curves over a field $k$. Furthermore, if Conjecture \ref{anyrank} given below is confirmed, then it is impossible to build a phantom as a semiorthogonal complement
to an exceptional collection of line bundles of maximal length in the derived category of a Fano toric DM stack $\mathbb{P}_{\mathbf{\Sigma}}$. Our main result in the paper shows this in the case of Picard rank less or equal to two.
%Ch. B$\ddot{\mathrm{o}}$hning, H-Ch. Graf von Bothmer, L. Katzarkov, and
%P. Sosna have successfully obtained a phantom as a semiorthogonal complement to an exceptional collection of maximal length
%on a Barlow surface.

\smallskip

The subject of exceptional collections on toric varieties and stacks has its own rich history. Kawamata constructed exceptional collections in the bounded derived categories of of coherent sheaves on smooth Deligne-Mumford stacks in \cite{Ka}. Alastair King conjectured in \cite{Ki} that every smooth toric variety has a full strong exceptional collection of line bundles. Although the conjecture was proved to be false in \cite{HP1}, rich and varied results related to the conjecture were proved in \cite{BHKing,HP2,M,E,I,N.Prabhu}. In particular, it was proved in \cite{BHKing} that there exist full strong exceptional collections of line bundles on smooth toric Fano DM stacks of Picard number no more than two and of any Picard number in dimension two.

\smallskip

The full strong exceptional collections of line bundles constructed in \cite{BHKing} have length equal the rank of the (Grothendieck) $K$-theory group, which is known to be necessary, see for example \cite{SOrlov}. It is natural to ask whether any strong exceptional collection of line bundles of this length is a full strong exceptional collection. That is to say that the subcategory generated by all elements in the strong collection equals $\mathbf{D}^b(coh(\mathbb{P}_{\mathbf{\Sigma}}))$, and there is no orthogonal complement phantom category. We propose the following conjecture.
\begin{conjecture}\label{anyrank}
Any strong exceptional collection of line bundles of maximal length on a Fano toric DM stack is a full strong exceptional collection.
\end{conjecture}
In this paper, we prove Conjecture \ref{anyrank} for $\mathrm{rk}(\mathrm{Pic} (\mathbb{P}_{\mathbf{\Sigma}}))=1$ (Theorem \ref{main1}) and $\mathrm{rk}(\mathrm{Pic} (\mathbb{P}_{\mathbf{\Sigma}}))=2$ (Theorem \ref{main2}).  Our main idea is to "shrink" the strong exceptional collection by moving some specific elements successively and eventually obtain a standard full strong exceptional collection given in \cite{BHKing}.

\smallskip

The paper is organized as follows. Section \ref{overview} recalls gives basic knowledge of toric DM stacks
and (strong) exceptional collection of line bundles on $\mathbb{P}_{\mathbf{\Sigma}}$. In Section \ref{1}, we prove Conjecture \ref{anyrank} for the case of
$\mathrm{rk}(\mathrm{Pic} (\mathbb{P}_{\mathbf{\Sigma}}))=1$. In Section \ref{2}, Conjecture \ref{anyrank} for the case of
the rank of $\mathrm{Pic} (\mathbb{P}_{\mathbf{\Sigma}})$ equals two is settled. Section \ref{comments} contains brief discussion of further directions.

\medskip

\noindent{\it Acknowledgements.} This work was prompted by a question of Shizhuo Zhang. Lev Borisov was partially supported by
NSF grant DMS-1601907.

\section{(Strong) exceptional collections of line bundles on toric Deligne-Mumford stacks}\label{overview}
In this section, we give an overview of toric DM stacks $\mathbb{P}_{\mathbf{\Sigma}}$, the corresponding Grothendieck group and (strong) exceptional collections of line bundles on $\mathbb{P}_{\mathbf{\Sigma}}$. Since all of this is well known, we try to be brief.

\smallskip

Let $\Sigma$ be a complete fan with $m$ one-dimensional cones in a lattice $N$ which is a free abelian group of finite rank. The assumption that $N$ has no torsion allows us to refrain from the technicalities of the derived Gale duality of \cite{BCS}. We pick a lattice point $v$ in each of the one-dimensional cones of $\Sigma$ and get a complete stacky fan $\mathbf{\Sigma}=(\Sigma, \{v_i\}_{i=1}^{m})$, see \cite{BCS}. The toric DM stack $\mathbb{P}_{\mathbf{\Sigma}}$ associated to the stacky fan $\mathbf{\Sigma}$ is constructed in \cite{BCS} as a stack version of the homogeneous coordinate ring construction of a toric variety \cite{C}. Line bundles on $\mathbb{P}_{\mathbf{\Sigma}}$ are described in \cite{BH,BHKing} similar to the scheme case of \cite{D,F}.

\begin{proposition}\label{pic}
The Picard group of $\mathbb{P}_{\mathbf{\Sigma}}$ is generated by $\{E_i\}_{i=1}^{m}$ with relations $\sum_{i=1}^{m}(w_i\cdot v_i)E_i$ for all $w$ in the character lattice $M=N^*$.
\end{proposition}
\begin{proof}
See \cite{BHKing}.
\end{proof}

\begin{definition}
An object $F$ in $\mathbf{D}^b(coh(\mathbb{P}_{\mathbf{\Sigma}}))$ is exceptional if $\mathrm{Hom}(F,F)=\mathbb{C}$ and $\mathrm{Ext}^t(F,F)=\mathrm{Hom}(F,F[t])=0$ for $t\neq 0$.
A sequence of exceptional objects $(F_1, F_2,\ldots,F_n)$ in $\mathbf{D}^b(coh(\mathbb{P}_{\mathbf{\Sigma}}))$ is called an exceptional collection if
$$\mathrm{Ext}^t(F_i,F_j)=\mathrm{Hom}(F_i,F_j[t])=0$$ for all $i>j$ and all $t\in \mathbb{Z}$. An exceptional collection is further called a strong exceptional collection if $$\mathrm{Ext}^t(F_i,F_j)=0$$ for all $i<j$ and all $t\in \mathbb{Z}\setminus \{0\}$.
\end{definition}

\begin{remark}\label{indexed}
A subset $\mathcal{T}$ of $\mathrm{Pic} (\mathbb{P}_{\mathbf{\Sigma}})$ can be indexed to form a strong exceptional collection if and only if $\mathrm{Ext}^t(\mathcal{L}_1,\mathcal{L}_2)=0$ for any $\{\mathcal{L}_1,\mathcal{L}_2\}\in\mathcal{T}$ and any $t>0$. The reason is that the existence of nonzero $\mathrm{Hom}(\mathcal{L}_1,\mathcal{L}_2)$ induces a partial order on the set $\mathcal{T}$ which can be extended to a linear order.
\end{remark}

\begin{definition}\label{defi.str}
\cite{BHKing} Let $\mathcal{T}$ be a finite set of line bundles on $\mathbb{P}_{\mathbf{\Sigma}}$ (which are always exceptional objects on $\mathbb{P}_{\mathbf{\Sigma}}$). We call $\mathcal{T}$ a full strong exceptional collection if $$\mathrm{Ext}^t(\mathcal{L}_1,\mathcal{L}_2)$$ for any $\{\mathcal{L}_1,\mathcal{L}_2\}\in\mathcal{T}$ and any $t>0$ and the derived category of $\mathbb{P}_{\mathbf{\Sigma}}$ is generated by the line bundles in $\mathcal{T}$.
\end{definition}

\begin{definition}
A toric DM stack $\mathbb{P}_{\mathbf{\Sigma}}$ is called Fano if the chosen points $v_i$ are precisely the vertices of a simplicial convex polytope in $N_{\mathbb{R}}$.
\end{definition}

\begin{definition}
\cite{BH} Let $\mathbb{P}_{\mathbf{\Sigma}}$ be a smooth DM stack.
The (Grothendieck) $K$-theory group $K_0(\mathbb{P}_{\mathbf{\Sigma}})$ is defined to be the quotient of the free
abelian group generated by coherent sheaves $\mathcal{F}$ on $\mathbb{P}_{\mathbf{\Sigma}}$ by the relations
$[\mathcal{F}_1]-[\mathcal{F}_2]+[\mathcal{F}_3]$ for all exact sequences $0 \rightarrow \mathcal{F}_1 \rightarrow \mathcal{F}_2 \rightarrow \mathcal{F}_3 \rightarrow 0$.
\end{definition}

\begin{lemma}\label{basis}
\cite{SOrlov} Let $\mathbb{P}_{\mathbf{\Sigma}}$ be a Fano toric DM stack and $(\mathcal{F}_1, \mathcal{F}_2,\ldots,\mathcal{F}_n)$ be an exceptional collection of objects in $\mathbf{D}^b(coh(\mathbb{P}_{\mathbf{\Sigma}}))$. If $n=\mathrm{rank}K_0(\mathbb{P}_{\mathbf{\Sigma}})$, then $\mathcal{F}_1, \mathcal{F}_2,\ldots,\mathcal{F}_n$ is a basis of $K_0(\mathbb{P}_{\mathbf{\Sigma}})$.
\end{lemma}

\begin{corollary}\label{<rkK}
Let $(\mathcal{F}_1, \mathcal{F}_2,\ldots,\mathcal{F}_n)$ be an exceptional collection of objects in $\mathbf{D}^b(coh(\mathbb{P}_{\mathbf{\Sigma}}))$. Then $n\leq \mathrm{rk}(K_0(\mathbb{P}_{\mathbf{\Sigma}}))$.
\end{corollary}
\section{The case of $\mathrm{rk}(\mathrm{Pic} (\mathbb{P}_{\mathbf{\Sigma}}))=1$}\label{1}
In the section, we prove Conjecture \ref{anyrank} when the rank of $\mathrm{Pic} (\mathbb{P}_{\mathbf{\Sigma}})$ is one.

\medskip

Let $\mathbb{P}_{\mathbf{\Sigma}}$ be a Fano toric DM stack such that $\mathrm{Pic} (\mathbb{P}_{\mathbf{\Sigma}})$ has no torsion and rank one. In this case $\mathbb{P}_{\mathbf{\Sigma}}$ is a weighted projective space which we denote by $W\mathbb{P}(w_1,\ldots,w_m)$, where $gcd(w_1,\ldots,w_m)=1$. \footnote{This condition comes from our assumption that $N$ has no torsion.} The rank of $K_0(\mathrm{Pic} (\mathbb{P}_{\mathbf{\Sigma}}))$ is $\sum_{i=1}^mw_i$. The Picard group $\mathrm{Pic} (\mathbb{P}_{\mathbf{\Sigma}})$ is $\{\mathcal{O}(d)| d\in \mathbb{Z}\}$, where $\mathcal{O}(E_i)=\mathcal{O}(w_i)$. By \cite{BHKing}, we know that $\mathbb{P}_{\mathbf{\Sigma}}$ possesses a full strong exceptional collection of line bundles.

\begin{proposition}\label{stronginBH}
\cite{BHKing} Let $\mathcal{T}=\{\mathcal{O}(w)|-\mathrm{rk}(K_0(\mathbb{P}_{\mathbf{\Sigma}}))+1\leq w \leq0\}$. Then $\mathcal{T}$ forms a full strong exceptional collection in the derived category of $W\mathbb{P}(w_1,\ldots,w_m)$.
\end{proposition}
\begin{proof}
See \cite{BHKing}.
\end{proof}

From \cite{BHKing}, for any $d_1, d_2\in\mathbb{Z}$, we know that
\begin{equation*}
\begin{split}
&\mathrm{Ext}^{\mathrm{rk}(N)}(\mathcal{O}(d_1),\mathcal{O}(d_2))\neq 0 \Leftrightarrow d_2-d_1=\sum_{i=1}^{m}a_iw_i, \text{ for some }a_i\in \mathbb{Z}_{<0};\\
&\mathrm{Hom}(\mathcal{O}(d_1),\mathcal{O}(d_2))\neq 0 \Leftrightarrow d_2-d_1=\sum_{i=1}^{m}a_iw_i, \text{ for some }a_i\in\mathbb{Z}_{\geq0}.
\end{split}
\end{equation*}

\begin{remark}\label{ec,sec}
In the case of $\mathrm{rk}(\mathrm{Pic} (\mathbb{P}_{\mathbf{\Sigma}}))=1$, any exceptional collection on $X=\mathbb{P}_{\mathbf{\Sigma}}$ is a strong exceptional collection. Indeed, let
$$\mathcal{T}=(\mathcal{O}(s_1),\ldots,\mathcal{O}(s_n))$$ be an exceptional collection on $\mathbb{P}_{\mathbf{\Sigma}}$. We have $\mathrm{Hom}(\mathcal{O}(s_j),\mathcal{O}(s_i))=0$ for $j>i$. Then
$\mathrm{Ext}^{\mathrm{rk}(N)}(\mathcal{O}(s_i),\mathcal{O}(s_j))=0$ for $j>i$. Otherwise, we get $s_j-s_i=\sum_{i=1}^{m}a_iw_i$, where $a_i\in \mathbb{Z}_{<0}$. This implies $s_j-s_i=\sum_{i=1}^{m}b_iw_i$, where $b_i=-a_i\in\mathbb{Z}_{\geq0}$, which contradicts $\mathrm{Hom}(\mathcal{O}(s_j),\mathcal{O}(s_i))=0$.
\end{remark}

{\bf Main idea.} Starting from an exceptional collection $\mathcal{T}$ of line bundles of maximal length, i.e., with $\sum_{i=1}^{m}w_i$ elements, we construct other exceptional collections
of maximal length in $\mathcal{D}(\mathcal{T})$, the subcategory generated by elements in $\mathcal{T}$. Eventually, we will get to the exceptional collection in Proposition \ref{stronginBH} given in \cite{BHKing}.
This allows us to conclude that $\mathcal{D}(\mathcal{T})=\mathbf{D}^b(coh(\mathbb{P}_{\mathbf{\Sigma}}))$.

\smallskip

The main step is to "move" the smallest element of the exceptional collection $\mathcal{T}$ by $\sum_{i=1}^mw_i$, see Figure \ref{fig:1}.
\begin{figure}[H]
  \includegraphics[width=0.96\textwidth]{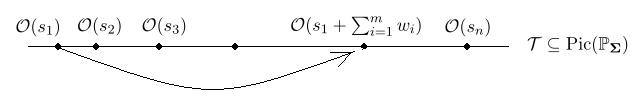}
  \caption{}
  \label{fig:1}
\end{figure}
Specifically: If line bundles $\mathcal{O}(s_1),\ldots,\mathcal{O}(s_n)$, where $s_1<s_2<\cdots<s_n$, form a strong exceptional collection $\mathcal{T}$ of maximal length,
then
\begin{enumerate}
\item $\mathcal{O}(s_1+\sum_{i=1}^mw_i)$ is not in the strong exceptional collection $\mathcal{T}$ (Lemma \ref{s_1+nnotin});
\item By replacing $\mathcal{O}(s_1)$ with $\mathcal{O}(s_1+\sum_{i=1}^mw_i)$ and reordering, we get another strong exceptional collection (Lemma \ref{anotersec});
\item $\mathcal{O}(s_1+\sum_{i=1}^mw_i)\in\mathcal{D}(\mathcal{T})$, so the new collection generates a subcategory of $\mathcal{D}(\mathcal{T})$ (Corollary \ref{s_1+n}).
\end{enumerate}

Once we know these that these moves are possible, we can "shrink" the exceptional collection to make it one from Propostion \ref{stronginBH} (Theorem \ref{main1}).

\begin{example}
We consider an exceptional collection on $W\mathbb{P}(5,6)$ $$(\mathcal{O}(-15),\mathcal{O}(-13),\mathcal{O}(-10), \mathcal{O}(-9),\mathcal{O}(-8),\mathcal{O}(-7),\mathcal{O}(-6),\mathcal{O}(-5),\mathcal{O}(-3),\mathcal{O}(-1),\mathcal{O})$$ of maximal length $11$. We replace $\mathcal{O}(-15)$ by $\mathcal{O}(-15+11)=\mathcal{O}(-4)$ to get another strong exceptional collection $$(\mathcal{O}(-13),\mathcal{O}(-10), \mathcal{O}(-9),\mathcal{O}(-8),\mathcal{O}(-7),\mathcal{O}(-6),\mathcal{O}(-5),\mathcal{O}(-4),\mathcal{O}(-3),\mathcal{O}(-1),\mathcal{O}).$$ Then we replace $\mathcal{O}(-13)$ by $\mathcal{O}(-13+11)=\mathcal{O}(-2)$ to get $$(\mathcal{O}(-10), \mathcal{O}(-9),\mathcal{O}(-8),\mathcal{O}(-7),\mathcal{O}(-6),\mathcal{O}(-5),\mathcal{O}(-4),\mathcal{O}(-3),\mathcal{O}(-2),\mathcal{O}(-1),\mathcal{O})$$ which is a full strong exceptional collection in Proposition \ref{stronginBH} given in \cite{BHKing}.
\end{example}

\begin{lemma}\label{s_1+nnotin}
Let $\mathcal{T}=\{\mathcal{O}(s_1),\ldots,\mathcal{O}(s_n)\}$ be a strong exceptional collection. Then $\mathcal{O}(s_1+\sum_{i=1}^mw_i)\notin \mathcal{T}$.
\end{lemma}
\begin{proof}
If $\mathcal{O}(s_1+\sum_{i=1}^mw_i)\in \mathcal{T}$, then $\mathrm{Ext}^{\mathrm{rk}(N)}(\mathcal{O}(s_1+\sum_{i=1}^mw_i),\mathcal{O}(s_1))\neq 0$ since $=s_1-(s_1+\sum_{i=1}^mw_i)=-\sum_{i=1}^nw_i$. This contradicts the assumption that $\mathcal{T}$ is a strong exceptional collection.
\end{proof}

\begin{lemma}\label{anotersec}
Let $\mathcal{T}=\{\mathcal{O}(s_1),\ldots,\mathcal{O}(s_n)\}$ be a strong exceptional collection of maximal length on $\mathbb{P}_{\mathbf{\Sigma}}$, where $s_1<s_2<\cdots<s_n$.
By replacing $\mathcal{O}(s_1)$ with $\mathcal{O}(s_1+\sum_{i=1}^mw_i)$ and reordering, we get another strong exceptional collection.
\end{lemma}
\begin{proof}
Let $\mathcal{T}^1$ be a collection obtained by replacing $\mathcal{O}(s_1)$ with $\mathcal{O}(s_1+\sum_{i=1}^mw_i)$. For any $i\in\{2,\ldots,n\}$, we have $s_i-s_1-\sum_{i=1}^mw_i>-\sum_{i=1}^mw_i$. Thus $\mathrm{Ext}^{\mathrm{rk}(N)}(\mathcal{O}(s_1+\sum_{i=1}^mw_i),\mathcal{O}(s_i))=0$. Also for any $i\in\{2,\ldots,n\}$, we have $\mathrm{Ext}^{\mathrm{rk}(N)}(\mathcal{O}(s_i),\mathcal{O}(s_1+\sum_{i=1}^mw_i))=0$. Otherwise, we get $s_1+\sum_{i=1}^mw_i-s_i=\sum_{i=1}^{m}a_iw_i$, where $a_i\leq -1$. Thus $s_1-s_i=\sum_{i=1}^mb_iw_i$, where $b_i< -1$. This implies $\mathrm{Ext}^{\mathrm{rk}(N)}(\mathcal{O}(s_i),\mathcal{O}(s_1))\neq0$, which contradicts the assumption that $\mathcal{T}$ is an exceptional collection.
\end{proof}

\begin{lemma}\label{move}
Let $\mathcal{T}=\{\mathcal{O}(s_1),\ldots,\mathcal{O}(s_n)\}$ be a strong exceptional collection of maximal length on $\mathbb{P}_{\mathbf{\Sigma}}$, where $s_1<s_2<\cdots<s_n$. Then $\mathcal{O}(s_1+\sum_{j\in J}w_j)$ is in $\mathcal{T}$ for any proper subset $J\subsetneqq \{1,2,\ldots,m\}$.
\end{lemma}
\begin{proof}
Let $s=s_1+\sum_{j\in J}w_j$.
We have $\mathrm{Ext}^{\mathrm{rk}(N)}(\mathcal{O}(s),\mathcal{O}(s_k))=0$ for all $k\in \{1,2,\ldots,n\}$. Otherwise, we have $s_k-s \in \sum_{i=1}^m\mathbb{Z}_{< 0}w_m$ for some $k$. However, we have $s_k-s_1\geq 0$. So $s_k-s=s_k-s_1-\sum_{j\in J}w_j > -\sum_{j=1}^{m}w_j$, which leads to contradiction.

\smallskip

We have $\mathrm{Ext}^{\mathrm{rk}(N)}(\mathcal{O}(s_k),\mathcal{O}(s))=0$ for all $k\in \{1,2,\ldots,n\}$. Otherwise, we get $s_1+\sum_{j\in J}w_j-s_k=s-s_k=\sum_{i=1}^{m}a_iw_i$ for some $k$, where $a_i\leq -1$. Thus $s_1-s_k=\sum_{i=1}^mb_iw_i$, where $b_i\leq -1$. Therefore $\mathrm{Ext}^{\mathrm{rk}(N)}(\mathcal{O}(s_k),\mathcal{O}(s_1))\neq0$, which contradicts that $\mathcal{T}$ is an exceptional collection.

\smallskip

If $\mathcal{O}(s)$ is not in $\mathcal{T}$, we can get another exceptional collection with $\sum_{i=1}^mw_i+1$ elements by inserting $\mathcal{O}(s)$ into $\mathcal{T}$. This is impossible by Corollary \ref{<rkK}.
\end{proof}

\begin{corollary}\label{s_1+n}
Let $\mathcal{T}=\{\mathcal{O}(s_1),\ldots,\mathcal{O}(s_n)\}$ be a strong exceptional collection of maximal length on $\mathbb{P}_{\mathbf{\Sigma}}$, where $s_1<s_2<\cdots<s_n$.
Then we have $\mathcal{O}(s_1+\sum_{i=1}^mw_i) \in \mathcal{D}(\mathcal{T})$.
\end{corollary}
\begin{proof}
We consider the Koszul complex \cite{BHKing}
$$0 \rightarrow \mathcal{O}(-\sum_{i=1}^mw_i) \rightarrow \cdots \rightarrow \bigoplus_{i=1}^m\mathcal{O}(-w_i)\rightarrow \mathcal{O} \rightarrow 0.$$
Then we tensor this complex by $\mathcal{O}(s_1+\sum_{i=1}^mw_i)$ and get
$$0 \rightarrow \mathcal{O}(s_1) \rightarrow \cdots \rightarrow \bigoplus_{i=1}^m\mathcal{O}(-\sum_{j\neq i}w_j+s_1)\rightarrow \mathcal{O}(s_1+\sum_{i=1}^mw_i) \rightarrow 0.$$
By Lemma \ref{move}, we have that $\mathcal{O}(s_1+\sum_{j\in J}w_j)$ is in $\mathcal{T}$ for any proper subset $J\subsetneqq \{1,2,\ldots,m\}$. Thus $\mathcal{O}(s_1+\sum_{i=1}^mw_i) \in \mathcal{D}(\mathcal{T})$.
\end{proof}

\begin{theorem}\label{main1}
Let $X=\mathbb{P}_{\mathbf{\Sigma}}$ be a Fano toric DM stack with $\mathrm{rank}(\mathrm{Pic} (\mathbb{P}_{\mathbf{\Sigma}}))=1$.
Assume $\mathcal{T}=\{\mathcal{O}(s_1),\ldots,\mathcal{O}(s_n)\}$ is a strong exceptional collection of maximal length. Then $\mathcal{T}$ is a full strong exceptional collection.
\end{theorem}
\begin{proof}
Without loss of generality, we assume $s_1<s_2<\cdots<s_n$.
If $s_1+\sum_{i=1}^mw_i > s_n$, then $\sum_{i=1}^mw_i>s_n-s_1$. Then $(s_1,\ldots,s_n)=(s_1,s_1+1,\ldots,s_1+\sum_{i=1}^mw_i)$. So $\mathcal{T}$ is a twist of the collection of \cite{BHKing} and is therefore full.
If $s_1+\sum_{i=1}^mw_i\leq s_n$, we get a new strong exceptional collection $$\mathcal{T}^1=\{\mathcal{O}(s_2),\ldots,\mathcal{O}(s_1+\sum_{i=1}^mw_i),\ldots,\mathcal{O}(s_n)\}$$ in $\mathcal{D}(\mathcal{T})$ by Lemma \ref{anotersec} and Corollary \ref{s_1+n}.

\smallskip

This process decreases $s_n-s_1$ and therefore terminates. So eventually we will be in the situation $s_1+\sum_{i=1}^mw_i>s_n$.
\end{proof}

\begin{remark}
When $\mathrm{Pic} (\mathbb{P}_{\mathbf{\Sigma}})$ has torsion, the arguments go without significant changes. The details are left to the reader.
\end{remark}
\section{The case of $\mathrm{rk}(\mathrm{Pic} (\mathbb{P}_{\mathbf{\Sigma}}))=2$}\label{2}
In this section, we consider Fano toric Deligne-Mumford stack $\mathbb{P}_{\mathbf{\Sigma}}$ associated to a stacky fan $\mathbf{\Sigma}=(\Sigma, \{v_i\}_{i=1}^{m})$ in the lattice $N$ with $\mathrm{rk}(N)=m-2$.  In this case, the rank of Picard group $\mathrm{rk}(\mathrm{Pic} (\mathbb{P}_{\mathbf{\Sigma}}))$ equals $2$.
Our aim is to prove Conjecture \ref{anyrank} in this case. We first assume that
$\mathrm{Pic} (\mathbb{P}_{\mathbf{\Sigma}})$ has no torsion for ease of exposition.

\smallskip

We recall the results of \cite{BHKing}.
\begin{proposition}\label{alpha}
\cite{BHKing} There exists a unique up to scaling collection of rational numbers $\alpha_i$ such that $\sum_{i=1}^m\alpha_i=0$ and $\sum_{i=1}^m\alpha_iv_i=0$. Moreover, all $\alpha_i$ in this relation are nonzero.
\end{proposition}
\begin{proof}
See \cite{BHKing}.
\end{proof}
We pick one such relation $\sum_{i=1}^m\alpha_iv_i=0$. Let $I_{+}=\{i|\alpha_i>0\}$ and $I_{-}=\{i|\alpha_i>0\}$. Then we have $\{1,\ldots,m\}=I_{+}\sqcup I_{-}$.
Let $E_{+}=\sum_{i\in I_{+}}(E_i)$ and $E_{-}=\sum_{i\in I_{-}}(E_i)$.
We consider a linear function $\alpha$ on $\mathrm{Pic}_{\mathbb{R}} (\mathbb{P}_{\mathbf{\Sigma}})$ with $\alpha(E_i)=\alpha_i$ from Proposition \ref{alpha}. Then $\alpha(E_{+})+\alpha(E_{-})=0$.

\smallskip

Moreover, from \cite{BHKing}, we can pick and fix a collection of positive numbers $r_i$, $i=1,\ldots,m$ such that $\sum_ir_i=1$ and $\sum_ir_iv_i=0$. This collection of positive numbers gives a linear function $f$ on $\mathrm{Pic}_{\mathbb{R}} (\mathbb{P}_{\mathbf{\Sigma}})$ with $f(E_i)=r_i>0$.

\smallskip

Let $P$ be a parallelogram in $\mathrm{Pic}_{\mathbb{R}} (\mathbb{P}_{\mathbf{\Sigma}})$ given by $$|f(x)|\leq \frac{1}{2}, |\alpha(x)|\leq \frac{1}{2}\sum_{i\in I_+}\alpha_i.$$ Pick a generic point $p\in \mathrm{Pic}_{\mathbb{R}} (\mathbb{P}_{\mathbf{\Sigma}})$ so that the lines along the sides of the parallelogram $p+P$ do not contain any points from $\mathrm{Pic}_{\mathbb{Q}} (\mathbb{P}_{\mathbf{\Sigma}})$. Then we have the following.

\begin{proposition}\label{p+P}
\cite{BHKing} The set $S$ of line bundles in $p+P$ forms a full strong exceptional collection on $\mathbb{P}_{\mathbf{\Sigma}}$.
\end{proposition}
\begin{proof}
See \cite{BHKing}.
\end{proof}

{\bf Notation:} The following notations will be used in our arguments. Let $\mathcal{T}=\{\mathcal{O}(D_1),\ldots,\mathcal{O}(D_n)\}$ be a collection of line bundles, we will abuse the notation slightly and denote by
$\mathrm{max}(\alpha(\mathcal{T}))$ the maximum value of $\alpha(D_i)$ for $O(D_i)$ in $\mathcal{T}$ (and similarly, for $\mathrm{min}$ and $f$). We denote $\mathcal{T}_{\mathrm{min}(f)}=\{D_i\in \mathcal{T}| f(D_i)=\mathrm{min}(f(\mathcal{T}))\}$.
%\begin{equation*}
%\begin{split}
%&\alpha(\mathcal{T})=\mathrm{max}_{D_i\in \mathcal{T} }\{\alpha(D_i)\}, \text{ }\text{ } \widetilde{\alpha}(\mathcal{T})=\mathrm{min}_{D_i\in \mathcal{T} %}\{\alpha(D_i)\},\\
%&f(\mathcal{T})=\mathrm{max}_{D_i\in \mathcal{T} }\{f(D_i)\},\text{ }\text{ } \widetilde{f}(\mathcal{T})=\mathrm{min}_{D_i\in \mathcal{T} }\{f(D_i)\},\\
%&\mathcal{T}_{\widetilde{f}}=\{D_i\in \mathcal{T}| f(D_i)=\widetilde{f}(\mathcal{T})\}.
%\end{split}
%\end{equation*}

\medskip

{\bf Main idea.} The idea of the proof is similar to the case $\mathrm{rk}(\mathrm{Pic} (\mathbb{P}_{\mathbf{\Sigma}}))=1$.
Starting from an exceptional collection $\mathcal{T}$ of line bundles of maximal length, we construct other exceptional collections
of maximal length in $\mathcal{D}(\mathcal{T})$, the subcategory generated by elements in $\mathcal{T}$. Eventually, we get to the exceptional collection in Proposition \ref{p+P}.

\medskip

{\bf Step $1$.} The first step is to "move" the largest elements in terms of the linear function $\alpha$ in the strong exceptional collection by $-E_+$ or $E_-$ to construct a new strong exceptional collection in $\mathcal{D}(\mathcal{T})$, see Figure \ref{fig:3}.
\begin{figure}[H]
  \includegraphics[width=0.48\textwidth]{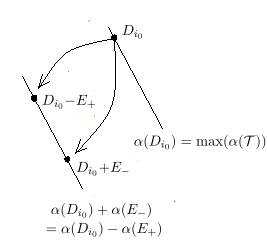}
  \caption{}
  \label{fig:3}
\end{figure}

\smallskip

Specifically: let $\mathcal{T}=(\mathcal{O}(D_1),\ldots,\mathcal{O}(D_n))$ be a strong exceptional collection of line bundles of maximal length. We pick $i_0\in \{1,\ldots,n\}$ such that $\alpha(D_{i_0})=\mathrm{max}(\alpha(\mathcal{T}))$.
Then
\begin{enumerate}
\item Both $\mathcal{O}(D_{i_0}-E_+)$ and $\mathcal{O}(D_{i_0}+E_-)$ are not in the strong exceptional collection $\mathcal{T}$ (Lemma \ref{E+E-notin});
\item Either replacing $\mathcal{O}(D_{i_0})$ with $\mathcal{O}(D_{i_0}-E_+)$ or with $\mathcal{O}(D_{i_0}+E_-)$, we get another strong exceptional collection after reordering (Lemma \ref{JL}, Lemma \ref{newstrong+} and Lemma \ref{newstrong-});
\item The new exceptional collection in $(2)$ is in $\mathcal{D}(\mathcal{T})$ (Lemma \ref{+in} and Lemma \ref{-in}).
\end{enumerate}

By repeating the above step (Theorem \ref{smallwidth}), we can reduce the problem to the strong exceptional collection $\mathcal{S}$ in $\mathcal{D}(\mathcal{T})$
such that all the line bundles in $\mathcal{S}$ are within a strip of width less than $\alpha(E_+)$, i.e., $\mathrm{max}({\alpha(\mathcal{S})})-\mathrm{min}(\alpha(\mathcal{S}))<\alpha(E_+)=\alpha(-E_-)$.

\smallskip

{\bf Step $2$.} From now on, we consider a strong exceptional collection $\mathcal{T}=(\mathcal{O}(D_1),\ldots,\mathcal{O}(D_n))$ of maximal length within a strip of width less than $\alpha(E_+)$. If $\mathrm{max}(f(\mathcal{T}))-\mathrm{min}(f(\mathcal{T}))<f(E_++E_-)=1$, then $\mathcal{T}$ is a full strong exceptional collection in
Proposition \ref{p+P}. This allows us to conclude that $\mathcal{D}(\mathcal{T})=\mathbf{D}^b(coh(\mathbb{P}_{\mathbf{\Sigma}}))$.

\smallskip

 Now, we assume $\mathrm{max}(f(\mathcal{T}))-\mathrm{min}(f(\mathcal{T}))\geq f(E_++E_-)=1$.
 We pick $j_0\in\{1,\ldots,n\}$ such that $\alpha(D_{j_0})=\mathrm{max}(\alpha(\mathcal{T}))$.
Then we can replace $\mathcal{O}(D_{i_0})$ with $\mathcal{O}(D_{i_0}-E_+)$ or $\mathcal{O}(D_{i_0}+E_-)$ to get another strong exceptional collection $\mathcal{T}'$ such that (Proposition \ref{move1}):
\begin{enumerate}
\item $\mathrm{max}(f(\mathcal{T}'))\leq \mathrm{max}(f(\mathcal{T}))$;
 \item  $\mathrm{min}(f(\mathcal{T}'))\geq\mathrm{min}(f(\mathcal{T}))$;
\item $\sharp(\mathcal{T}'_{\mathrm{min}(f)})\leq \sharp(\mathcal{T}_{\mathrm{min}(f)})$ if $\mathrm{min}(f(\mathcal{T}'))= \mathrm{min}(f(\mathcal{T}))$;
\item $\sharp(\{D_i\in \mathcal{T}'| f(D_i)=\mathrm{min}(f(\mathcal{T}))\})< \sharp(\mathcal{T}_{\mathrm{min}(f)})$ if $f(D_{i_0})=\mathrm{min}(f(\mathcal{T}))$.
\end{enumerate}

By repeating the above step (Theorem \ref{main2}), we get a new strong exceptional collection $\mathcal{S}$ such that $\mathrm{max}({\alpha(\mathcal{S})})-\mathrm{min}({\alpha(\mathcal{S})})<\alpha(E_+)=\alpha(-E_-)$ and $\mathrm{max}({f(\mathcal{S})})-\mathrm{min}({f(\mathcal{S})})<f(E_++E_-)=1$ which is one
 in Proposition \ref{p+P}. This allows us to conclude that $\mathcal{D}(\mathcal{T})=\mathbf{D}^b(coh(\mathbb{P}_{\mathbf{\Sigma}}))$.

\medskip

{\bf Details of proof.}
For a divisor class $D$ in $\mathrm{Pic} (\mathbb{P}_{\mathbf{\Sigma}})$, we write $D=\sum_{i\in I}(\geq 0)E_i$ if $D$ can be written as $D=\sum_{i\in I}a_iE_i$ with $a_i \in \mathbb{Z}_{\geq 0}$ for all $i$ in a subset $I\subseteq \{1,\ldots,m\}$. We use similar notation for other inequalities.
%For any subset $I\subseteq \{1,\ldots,m\}$, we denote $E_{I}=\sum_{i\in I}E_i$.
%Let a divisor $D=\sum_{i\in I}a_iE_i$, we briefly write
%\begin{equation*}
%\begin{split}
%&D=\sum_{i\in I}(\geq 0)E_i \text{ if all } a_i \in \mathbb{Z}_{\geq 0}; D=\sum_{i\in I}(\leq 0)E_i \text{ if all } a_i \in \mathbb{Z}_{\leq 0};\\
%&D=\sum_{i\in I}(<0)E_i \text{ if all } a_i \in \mathbb{Z}_{<0}; D=\sum_{i\in I}(>0)E_i \text{ if all } a_i \in \mathbb{Z}_{>0}.
%\end{split}
%\end{equation*}

\smallskip

The nonzero $\mathrm{Ext}$ groups between line bundles have been calculated in \cite{BHKing}. We denote by $\mathrm{Ext}^{+}$, $\mathrm{Ext}^{-}$ the groups associated to sets $I_+$, $I_-$. Specifically,
for any $D_1, D_2\in \mathrm{Pic} (\mathbb{P}_{\mathbf{\Sigma}})$, we have
\begin{equation*}
\begin{split}
&\mathrm{Ext}^{\mathrm{rk}(N)}(\mathcal{O}(D_1),\mathcal{O}(D_2))\neq 0 \Leftrightarrow D_2-D_1=\sum_{i\in \{1,\ldots,m\}}(<0)E_i;\\
&\mathrm{Ext}^{+}(\mathcal{O}(D_1),\mathcal{O}(D_2))\neq 0 \Leftrightarrow D_2-D_1=\sum_{i\in I_-}(< 0)E_i+\sum_{i\in I_+}(\geq 0)E_i; \\
&\mathrm{Ext}^{-}(\mathcal{O}(D_1),\mathcal{O}(D_2))\neq 0 \Leftrightarrow D_2-D_1=\sum_{i\in I_+}(< 0)E_i+\sum_{i\in I_-}(\geq 0)E_i;\\
&\mathrm{Hom}(\mathcal{O}(D_1),\mathcal{O}(D_2))\neq 0 \Leftrightarrow D_2-D_1=\sum_{i\in\{1,\ldots,m\}}(\geq 0)E_i.
\end{split}
\end{equation*}

\begin{lemma}\label{E+E-notin}
Let $\mathcal{T}=(\mathcal{O}(D_1),\ldots,\mathcal{O}(D_n))$ be a strong exceptional collection of line bundles on $\mathbb{P}_{\mathbf{\Sigma}}$.
If $i_0\in \{1,\ldots,n\}$, then both $\mathcal{O}(D_{i_0}-E_+)$ and $\mathcal{O}(D_{i_0}+E_-)$ are not in $\mathcal{T}$.
\end{lemma}
\begin{proof}
If $\mathcal{O}(D_{i_0}-E_+)\in \mathcal{T}$, we have $\mathrm{Ext}^{-}(\mathcal{O}(D_{i_0}), \mathcal{O}(D_{i_0}-E_+))\neq 0$ since $D_{i_0}-E_+-D_{i_0}=-E_+$.
If $\mathcal{O}(D_{i_0}+E_-)\notin \mathcal{T}$, we have $\mathrm{Ext}^{+}(\mathcal{O}(D_{i_0}+E_-), \mathcal{O}(D_{i_0})\neq 0$ since $D_{i_0}-D_{i_0}-E_-=-E_-$.
These contradict that $\mathcal{T}$ is a strong exceptional collection.
\end{proof}
%It is meaningful to check whether $\mathcal{O}(D_{i_0}-E_+)$ is in $\mathcal{D}(\mathcal{T})$.
%If $\mathcal{O}(D_{i_0}-E_+) \in \mathcal{D}(\mathcal{T})$, we can replace $D_{i_0}$ by $D_{i_0}-E_+$
%to get a new collection. Then the subcategory generated by all the elements in the new collection is still contained in $\mathcal{D}(\mathcal{T})$.
For any subset $I\subseteq \{1,\ldots,m\}$, we denote $E_{I}=\sum_{i\in I}E_i$.

\begin{lemma}\label{J}
Let $\mathcal{T}=\{\mathcal{O}(D_1),\ldots,\mathcal{O}(D_n)\}$ be a strong exceptional collection of line bundles on $\mathbb{P}_{\mathbf{\Sigma}}$.
We pick $i_0\in \{1,\ldots,n\}$ such that $\alpha(D_{i_0})=\mathrm{max}(\alpha(\mathcal{T}))$. Then for any proper subset $J$ of $I_{+}$ and any $k \in \{1,\ldots,n\}$, we have
\begin{equation*}
\begin{split}
&\mathrm{Ext}^*(\mathcal{O}(D_{i_0}-E_{J}), \mathcal{O}(D_k))=0, \text{ where }*=\mathrm{rk}(N),+,-;\\
&\mathrm{Ext}^*(\mathcal{O}(D_k), \mathcal{O}(D_{i_0}-E_{J}))=0, \text{ where }*=+,-.\\
\end{split}
\end{equation*}
\end{lemma}
\begin{proof}
$(1)$ We have $\mathrm{Ext}^{\mathrm{rk}(N)}(\mathcal{O}(D_{i_0}-E_{J}), \mathcal{O}(D_k))=0$. Otherwise, we get $D_k-D_{i_0}+E_{J}=\sum_{i\in \{1,\ldots,m\}}(<0)E_i$. Thus $D_k-D_{i_0}=\sum_{i\in \{1,\ldots,m\}}(<0)E_i-E_{J}=\sum_{i\in \{1,\ldots,m\}}(<0)E_i$. This implies $\mathrm{Ext}^{\mathrm{rk}(N)}(\mathcal{O}(D_{i_0}), \mathcal{O}(D_k))\neq 0$ which contradicts the assumption that $\mathcal{T}$ is a strong exceptional collection.

\smallskip

$(2)$  We have $\mathrm{Ext}^{+}(\mathcal{O}(D_{i_0}-E_{J}), \mathcal{O}(D_k))=0$. Otherwise, we get $D_k-D_{i_0}+E_{J}=\sum_{i\in I_-}(< 0)E_i+\sum_{i\in I_+}(\geq 0)E_i$. So $D_k-D_{i_0}=-E_--E_J+\sum_{i\in I_-}(\leq 0)E_i+\sum_{i\in I_+}(\geq 0)E_i$.
We have $\alpha(-E_-)=\alpha(E_+)>\alpha(E_J)$ since $J\subsetneqq I_{+}$. Also, $\alpha(\sum_{i\in I_-}(\leq 0)E_i)\geq 0$ and $\alpha(\sum_{i\in I_+}(\geq 0)E_i)\geq 0$. Thus $\alpha(D_k-D_{i_0})>0$ which contradicts the assumption that $\alpha(D_{i_0})=\mathrm{max}(\alpha(\mathcal{T}))$.

\smallskip

$(3)$ We have $\mathrm{Ext}^{-}(\mathcal{O}(D_{i_0}-E_{J}), \mathcal{O}(D_k))=0$. Otherwise, we have $D_k-D_{i_0}+E_{J}=\sum_{i\in I_+}(< 0)E_i+\sum_{i\in I_-}(\geq 0)E_i$. Thus $D_k-D_{i_0}=\sum_{i\in I_+}(< 0)E_i-E_J+\sum_{i\in I_-}(\geq 0)E_i=\sum_{i\in I_+}(< 0)E_i+\sum_{i\in I_-}(\geq 0)E_i$. This implies $\mathrm{Ext}^{-}(\mathcal{O}(D_{i_0}), \mathcal{O}(D_k))\neq 0$, contradiction.

\smallskip

$(4)$ We have $\mathrm{Ext}^+(\mathcal{O}(D_k), \mathcal{O}(D_{i_0}-E_{J}))=0$. Otherwise, we have $D_{i_0}-E_{J}-D_k=\sum_{i\in I_+}(\geq 0)E_i+\sum_{i\in I_-}(< 0)E_i$. Thus $D_{i_0}-D_k=\sum_{i\in I_+}(\geq 0)E_i-\sum_{i\in I_-}(< 0)E_i+E_J=\sum_{i\in I_+}(\geq 0)E_i-\sum_{i\in I_-}(< 0)E_i$. This implies $\mathrm{Ext}^{+}(\mathcal{O}(D_k), \mathcal{O}(D_{i_0}))\neq 0$, contradiction.

\smallskip

$(5)$ We have $\mathrm{Ext}^-(\mathcal{O}(D_k), \mathcal{O}(D_{i_0}-E_{J}))=0$. Otherwise, we have $D_{i_0}-E_{J}-D_k=\sum_{i\in I_+}(< 0)E_i+\sum_{i\in I_-}(\geq 0)E_i$. Thus $D_{i_0}-D_k=\sum_{i\in I_+}(< 0)E_i+E_J+\sum_{i\in I_-}(\geq 0)E_i$. We get $\alpha(\sum_{i\in I_+}(< 0)E_i)=\sum_{i\in I_+}(< 0)\alpha_i\leq \sum_{i\in I_+}(-1)\alpha_i< \sum_{i\in J}(-1)\alpha_i=\alpha(-E_J)$ since $J\subsetneqq I_{+}$. So $\alpha(\sum_{i\in I_+}(< 0)E_i+E_J)<0$. Also, $\alpha(\sum_{i\in I_-}(\geq 0)E_i)\leq 0$. This implies $\alpha(D_{i_0}-D_k)<0$ which contradicts the assumption that $\alpha(D_{i_0})=\mathrm{max}(\alpha(\mathcal{T}))$.
\end{proof}

\begin{lemma}\label{L}
Let $\mathcal{T}=\{\mathcal{O}(D_1),\ldots,\mathcal{O}(D_n)\}$ be a strong exceptional collection of line bundles on $\mathbb{P}_{\mathbf{\Sigma}}$.
We pick $i_0\in \{1,\ldots,n\}$ such that $\alpha(D_{i_0})=\mathrm{max}(\alpha(\mathcal{T}))$. Then for any proper subset $L$ of $I_{-}$ and any $j \in \{1,\ldots,n\}$, we have
\begin{equation*}
\begin{split}
&\mathrm{Ext}^*(\mathcal{O}(D_{i_0}+E_{L}), \mathcal{O}(D_j))=0, \text{ where }*=+,-;\\
&\mathrm{Ext}^*(\mathcal{O}(D_j), \mathcal{O}(D_{i_0}+E_{L}))=0, \text{ where }*=\mathrm{rk}(N),+,-;\\
\end{split}
\end{equation*}
\end{lemma}
\begin{proof}
The proof is analogous to the proof of Lemma \ref{J} and is left to the reader.
\end{proof}
Note that Lemmas \ref{J}, \ref{L} only cover vanishing of five out of possible six $\mathrm{Ext}^{>0}$ spaces. The next Lemma addresses the remaining space.
\begin{lemma}\label{JL}
Let $\mathcal{T}=\{\mathcal{O}(D_1),\ldots,\mathcal{O}(D_n)\}$ be a strong exceptional collection of line bundles on $\mathbb{P}_{\mathbf{\Sigma}}$. We pick $i_0\in \{1,\ldots,n\}$ such that $\alpha(D_{i_0})=\mathrm{max}(\alpha(\mathcal{T}))$. Then either $\mathrm{Ext}^{\mathrm{rk}(N)}(D_k, D_{i_0}-E_{J})=0$ for all $k\in\{1,\ldots,n\}$ and all $J\subseteq I_{+}$ or $\mathrm{Ext}^{\mathrm{rk}(N)}(D_{i_0}+E_{L}, D_j)=0$ for all $j\in\{1,\ldots,n\}$ and all $L\subseteq I_{-}$, or both.
\end{lemma}
\begin{proof}
If $\mathrm{Ext}^{\mathrm{rk}(N)}(\mathcal{O}(D_k),  \mathcal{O}(D_{i_0}-E_J))\neq 0$ for some $k$ and some $J\subseteq I_{+}$, then $$D_{i_0}-D_k-E_J=\sum(<0)E_i=-E_-+\sum_{I^-}(\leq 0)E_i+\sum_{I^+}(< 0)E_i.$$
If $\mathrm{Ext}^{\mathrm{rk}(N)}(\mathcal{O}(D_{i_0}+E_L),  \mathcal{O}(D_j))\neq 0$ for some $j$ and some $L\subseteq I_{-}$, then $$D_j-D_{i_0}-E_L=\sum(<0)E_i=-E_++\sum_{I^+}(\leq 0)E_i+\sum_{I^-}(< 0)E_i.$$
We add the two equations to get $$D_j-D_k-E_J-E_L=-E_+-E_-+\sum_{I^+}(< 0)E_i+\sum_{I^-}(< 0)E_i.$$
Thus
\begin{equation*}
\begin{split}
D_j-D_k &=(-E_++E_J)+(-E_-+E_L)+\sum_{I^+}(< 0)E_i+\sum_{I^-}(< 0)E_i\\
&=\sum_{I^+}(< 0)E_i+\sum_{I^-}(< 0)E_i
\end{split}
 \end{equation*}
 since $J\subseteq I_{+}$ and $L\subseteq I_{-}$. This implies $\mathrm{Ext}^{\mathrm{rk}(N)}(\mathcal{O}(D_k),  \mathcal{O}(D_j))\neq 0$ which contradicts the assumption that $\mathcal{T}$ is a strong exceptional collection.
\end{proof}

\begin{lemma}\label{+in}
Let $\mathcal{T}=\{\mathcal{O}(D_1),\ldots,\mathcal{O}(D_n)\}$ be a strong exceptional collection of maximal length on $\mathbb{P}_{\mathbf{\Sigma}}$. We pick $i_0\in \{1,\ldots,n\}$ such that $\alpha(D_{i_0})=\mathrm{max}(\alpha(\mathcal{T}))$.
Assume $\mathrm{Ext}^{\mathrm{rk}(N)}(\mathcal{O}(D_k), \mathcal{O}(D_{i_0}-E_{J}))=0$ for all $k\in\{1,\ldots,n\}$ and all proper subsets $J\subsetneqq I_+$. Then $\mathcal{O}(D_{i_0}-E_{+})\in \mathcal{D}(\mathcal{T})$.
\end{lemma}
\begin{proof}
We have $\mathcal{O}(D_{i_0}-E_{J}) \in \mathcal{T}$ for all $J\subsetneqq I_{+}$. Otherwise, there is $J\subsetneqq I_{+}$ such that $\mathcal{O}(D_{i_0}-E_{J}) \notin \mathcal{T}$. By Lemma \ref{J}, we can add $\mathcal{O}(D_{i_0}-E_{J})$ to $\mathcal{T}$ to get a strong
exceptional collection with more than $\mathrm{rk}(K_0(\mathbb{P}_{\mathbf{\Sigma}}))$ elements. This is impossible by Corollary \ref{<rkK}.

\smallskip

Now we consider the Koszul complex
$$0 \rightarrow \mathcal{O}(-E_+) \rightarrow \cdots \rightarrow \bigoplus_{i\in I_+}\mathcal{O}(-E_i)\rightarrow \mathcal{O} \rightarrow 0.$$
We tensor the complex by $\mathcal{O}(D_{i_0})$ to get
$$0 \rightarrow \mathcal{O}(D_{i_0}-E_+) \rightarrow \cdots \rightarrow \bigoplus_{i\in I_+}\mathcal{O}(-E_i+D_{i_0})\rightarrow \mathcal{O}(D_{i_0}) \rightarrow 0.$$
Since $\mathcal{O}(D_{i_0}-E_{J}) \in \mathcal{T}$ for all $J\subsetneqq I_{+}$, we get $\mathcal{O}(D_{i_0}-E_+) \in \mathcal{D}(\mathcal{T})$.
\end{proof}

\begin{lemma}\label{-in}
Let $\mathcal{T}=\{\mathcal{O}(D_1),\ldots,\mathcal{O}(D_n)\}$ be a strong exceptional collection of maximal length on $\mathbb{P}_{\mathbf{\Sigma}}$. We pick $i_0\in \{1,\ldots,n\}$ such that $\alpha(D_{i_0})=\mathrm{max}(\alpha(\mathcal{T}))$.
Assume $\mathrm{Ext}^{\mathrm{rk}(N)}(\mathcal{O}(D_{i_0}+E_{L}), \mathcal{O}(D_j))=0$ for any $j\in\{1,\ldots,n\}$ for any subset $L\subsetneqq I_-$. Then $\mathcal{O}(D_{i_0}+E_{-})\in \mathcal{D}(\mathcal{T})$.
\end{lemma}
\begin{proof}
Analogous to Lemma \ref{+in}.
\end{proof}

\begin{lemma}\label{newstrong+}
Let $\mathcal{T}=\{\mathcal{O}(D_1),\ldots,\mathcal{O}(D_n)\}$ be a strong exceptional collection of line bundles of maximal length on $\mathbb{P}_{\mathbf{\Sigma}}$.
We pick $i_0\in \{1,\ldots,n\}$ such that $\alpha(D_{i_0})=\mathrm{max}(\alpha(\mathcal{T}))$.
Assume $\mathrm{Ext}^{\mathrm{rk}(N)}(\mathcal{O}(D_k), \mathcal{O}(D_{i_0}-E_{J}))=0$ for any $k\in\{1,\ldots,n\}$ and any subset $J\subseteq I_+$. Then we can get a new strong exceptional collection by replacing $\mathcal{O}(D_{i_0})$ with $\mathcal{O}(D_{i_0}-E_{+})$ and reordering.
\end{lemma}
\begin{proof}
We will carefully check vanishing of all six $\mathrm{Ext}^{>0}$ spaces with the new element of the collection.

\smallskip

$(1)$ We have $\mathrm{Ext}^{\mathrm{rk}(N)}(\mathcal{O}(D_{i_0}-E_+), \mathcal{O}(D_k))=0$ by the same argument as in $(1)$ of Lemma \ref{J}.

\smallskip

$(2)$ We have $\mathrm{Ext}^{+}(\mathcal{O}(D_{i_0}-E_+), \mathcal{O}(D_k))=0$. Otherwise, we get $D_k-D_{i_0}+E_{+}=\sum_{i\in I_-}(< 0)E_i+\sum_{i\in I_+}(\geq 0)E_i$. So $D_k-D_{i_0}=-E_--E_++\sum_{i\in I_-}(\leq 0)E_i+\sum_{i\in I_+}(\geq 0)E_i$.
We have $\alpha(-E_--E_+)=0$. Also, the coefficients in $\sum_{i\in I_-}(\leq 0)E_i+\sum_{i\in I_+}(\geq 0)E_i$ cannot be all zero. Otherwise, we have $D_k-D_{i_0}=-E_--E_+$. This implies $\mathrm{Ext}^{\mathrm{rk}(N)}(D_{i_0}, D_k)\neq0$ which contradicts that $\mathcal{T}$ is a strong exceptional collection. Now we get $\alpha(\sum_{i\in I_-}(\leq 0)E_i+\sum_{i\in I_+}(\geq 0)E_i)> 0$. Thus $\alpha(D_k-D_{i_0})>0$ which contradicts the assumption that $\alpha(D_{i_0})=\mathrm{max}(\alpha(\mathcal{T}))$.

\smallskip

$(3)$ We have $\mathrm{Ext}^{-}(\mathcal{O}(D_{i_0}-E_{+}), \mathcal{O}(D_k))=0$ by the same argument as in $(3)$ of Lemma \ref{J}.

\smallskip

$(4)$ By assumption, $\mathrm{Ext}^{\mathrm{rk}(N)}(\mathcal{O}(D_k), \mathcal{O}(D_{i_0}-E_+))=0$ for all $k\in\{1,\ldots,n\}$.

\smallskip

$(5)$ We have $\mathrm{Ext}^+(\mathcal{O}(D_k), \mathcal{O}(D_{i_0}-E_+))=0$ by the same argument as in $(4)$ of Lemma \ref{J}.

\smallskip

$(6)$ We have $\mathrm{Ext}^-(\mathcal{O}(D_k), \mathcal{O}(D_{i_0}-E_{+}))=0$. Otherwise, we have $D_{i_0}-E_{+}-D_k=\sum_{i\in I_+}(< 0)E_i+\sum_{i\in I_-}(\geq 0)E_i$. Thus $D_{i_0}-D_k=\sum_{i\in I_+}(< 0)E_i+E_++\sum_{i\in I_-}(\geq 0)E_i$. If one of the coefficients in $\sum_{i\in I_+}(< 0)E_i$ is less than $-1$, then $\alpha(\sum_{i\in I_+}(< 0)E_i)=\sum_{i\in I_+}(< 0)\alpha_i< \sum_{i\in I_+}(-1)\alpha_i=\alpha(-E_+)$. So $\alpha(\sum_{i\in I_+}(< 0)E_i+E_+)< 0$. If all the coefficients in $\sum_{i\in I_+}(< 0)E_i$ equal $-1$, then $D_{i_0}-D_k=\sum_{i\in I_-}(\geq 0)E_i$. Since $D_{i_0}\neq D_k$, the coefficients in $\sum_{i\in I_-}(\geq 0)E_i$ cannot be all zero. Thus $\alpha(\sum_{i\in I_-}(\geq 0)E_i)<0$. Now, we obtain that either $\alpha(\sum_{i\in I_+}(< 0)E_i+E_+)< 0$ or $\alpha(\sum_{i\in I_-}(\geq 0)E_i)<0$. Therefore $\alpha(D_{i_0}-D_k)=\alpha(\sum_{i\in I_+}(< 0)E_i+E_++\sum_{i\in I_-}(\geq 0)E_i)<0$ which contradicts the assumption that $\alpha(D_{i_0})=\mathrm{max}(\alpha(\mathcal{T}))$.

\smallskip

We have verified that there are no $\mathrm{Ext}^{>0}$ spaces between the new element and other elements of the collection.
\end{proof}

\begin{lemma}\label{newstrong-}
Let $\mathcal{T}=\{\mathcal{O}(D_1),\ldots,\mathcal{O}(D_n)\}$ be a strong exceptional collection of line bundles of maximal length on $\mathbb{P}_{\mathbf{\Sigma}}$.
We pick $i_0\in \{1,\ldots,n\}$ such that $\alpha(D_{i_0})=\mathrm{max}(\alpha(\mathcal{T}))$.
Assume $\mathrm{Ext}^{\mathrm{rk}(N)}(\mathcal{O}(D_{i_0}+E_{L}), \mathcal{O}(D_j))=0$ for any $j\in\{1,\ldots,n\}$ for any subset $L\subseteq I_-$. Then we can get a new strong exceptional collection in $\mathcal{D}(\mathcal{T})$ by replacing $\mathcal{O}(D_{i_0})$ with $\mathcal{O}(D_{i_0}+E_{-})$ and reordering.
\end{lemma}
\begin{proof}
Analogous to Lemma \ref{newstrong+}.
\end{proof}

\begin{proposition}\label{smallwidth}
Let $\mathcal{T}=\{\mathcal{O}(D_1),\ldots,\mathcal{O}(D_n)\}$ be a strong exceptional collection of line bundles of maximal length on $\mathbb{P}_{\mathbf{\Sigma}}$.  We can construct a new strong exceptional collection $\mathcal{S}$ in $\mathcal{D}(\mathcal{T})$ such that $\mathrm{max}(\alpha(\mathcal{S}))-\mathrm{min}(\alpha(\mathcal{S}))<\alpha(E_+)=\alpha(-E_-)$.
\end{proposition}
\begin{proof}
The argument is similar to that of Theorem \ref{main1}.
Let $\alpha(D_{i_0})=\mathrm{max}(\alpha(\mathcal{T}))$.
By Lemma \ref{JL}, we have either $\mathrm{Ext}^{\mathrm{rk}(N)}(D_k, D_{i_0}-E_{J})=0$ for any $k\in\{1,\ldots,m\}$ and any $J\subseteq I_{+}$ or $\mathrm{Ext}^{\mathrm{rk}(N)}(D_{i_0}+E_{L}, D_j)=0$ for any $j\in\{1,\ldots,m\}$ and any $L\subseteq I_{-}$. By Lemma \ref{+in}, Lemma \ref{-in}, Lemma \ref{newstrong+} and Lemma \ref{newstrong-}, we get a new strong exceptional collection $\mathcal{T}'$ in $\mathcal{D}(\mathcal{T})$ by replacing $\mathcal{O}(D_{i_0})$ by $\mathcal{O}(D_{i_0}-E_+)$ or $\mathcal{O}(D_{i_0}+E_-)$, and reordering. See Figure \ref{fig:3}.

\smallskip

We have $\mathrm{max}(\alpha(\mathcal{T}'))\leq\mathrm{max}(\alpha(\mathcal{T}))$ since $\alpha(D_{i_0}-E_+)=\alpha(D_{i_0}+E_-)<\alpha(D_{i_0})=\mathrm{max}(\alpha(\mathcal{T}))$.
After a finite number of steps, we replace successively all $\mathcal{O}(D_{i})$ such that $\alpha(D_{i})=\mathrm{max}(\alpha(\mathcal{T}))$ by $\mathcal{O}(D_{i}-E_+)$ or $\mathcal{O}(D_{i}+E_-)$ to get a new strong exceptional collection $\mathcal{T}^1$ in $\mathcal{D}(\mathcal{T})$ such that $\mathrm{max}(\alpha(\mathcal{T}^1))<\mathrm{max}(\alpha(\mathcal{T}))$.

\smallskip

If $\mathrm{min}(\alpha(\mathcal{T}^1))<\mathrm{min}(\alpha(\mathcal{T}))$, there exists some $D_{}$ such that $\alpha(D_{i})=\mathrm{max}(\alpha(\mathcal{T}))$ and $\alpha(D_{i}\mp E_{\pm})=\mathrm{min}(\alpha(\mathcal{T}))$. Now we have $$\mathrm{max}(\alpha(\mathcal{T}^1))-\mathrm{min}(\alpha(\mathcal{T}^1))<\mathrm{max}(\alpha(\mathcal{T}))-\mathrm{min}(\alpha(\mathcal{T}^1))=\alpha(D_{i})-\alpha(D_{i}\mp E_{\pm})=\alpha(E_+).$$

\smallskip

If $\mathrm{min}(\alpha(\mathcal{T}^1))\geq\mathrm{min}(\alpha(\mathcal{T}))$, then $\mathrm{max}(\alpha(\mathcal{T}^1))-\mathrm{min}(\alpha(\mathcal{T}^1))<\mathrm{max}(\alpha(\mathcal{T}))-\mathrm{min}(\alpha(\mathcal{T}))$.

\smallskip

This process decreases $\mathrm{max}(\alpha(\mathcal{T}))-\mathrm{min}(\alpha(\mathcal{T}))$. Eventually we will be in the situation that $\mathrm{max}(\alpha(\mathcal{T}))-\mathrm{min}(\alpha(\mathcal{T}))<\alpha(E_+)=\alpha(-E_-)$.
\end{proof}

\begin{proposition}\label{move1}
Let $\mathcal{T}=\{\mathcal{O}(D_1),\ldots,\mathcal{O}(D_n)\}$ be a strong exceptional collection of line bundles with length $n=\mathrm{rk}(K_0(\mathbb{P}_{\mathbf{\Sigma}}))$.
 Assume $$\mathrm{max}(f(\mathcal{T}))-\mathrm{min}(f(\mathcal{T}))\geq f(E_++E_-)=1.$$
 We pick $i_0\in\{1,\ldots,n\}$ such that $\alpha(D_{i_0})=\mathrm{max}(\alpha(\mathcal{T}))$.
Then we can replace $\mathcal{O}(D_{i_0})$ with $\mathcal{O}(D_{i_0}-E_+)$ or $\mathcal{O}(D_{i_0}+E_-)$ to get another strong exceptional collection $\mathcal{T}'$ in $\mathcal{D}(\mathcal{T})$ such that:
\begin{enumerate}
\item $\mathrm{max}(f(\mathcal{T}'))\leq \mathrm{max}(f(\mathcal{T}))$;
 \item  $\mathrm{min}(f(\mathcal{T}'))\geq\mathrm{min}(f(\mathcal{T}))$;
\item $\sharp(\mathcal{T}'_{\mathrm{min}(f)})\leq \sharp(\mathcal{T}_{\mathrm{min}(f)})$ if $\mathrm{min}(f(\mathcal{T}'))= \mathrm{min}(f(\mathcal{T}))$;
\item $\sharp(\{D_i\in \mathcal{T}'| f(D_i)=\mathrm{min}(f(\mathcal{T}))\})< \sharp(\mathcal{T}_{\mathrm{min}(f)})$ if $f(D_{i_0})=\mathrm{min}(f(\mathcal{T}))$.
\end{enumerate}
\end{proposition}
\begin{proof}
If
\begin{equation}\label{fine}
\mathrm{min}(f(\mathcal{T}))<f(D_{i_0}-E_+)<f(D_{i_0}+E_-)\leq  \mathrm{max}(f(\mathcal{T})),
\end{equation}
by Lemma \ref{JL}, Lemma \ref{newstrong+} and Lemma \ref{newstrong-}, we can replace $\mathcal{O}(D_{i_0})$ with $\mathcal{O}(D_{i_0}-E_+)$ or $\mathcal{O}(D_{i_0}+E_-)$ to reach the result.

\smallskip

If Equation \ref{fine} fails, there are several cases to consider.

\smallskip

 {\bf Case }$f(D_{i_0}-E_+)\leq\mathrm{min}(f(\mathcal{T}))$. We have $f(D_{i_0}+E_-)\leq \mathrm{max}(f(\mathcal{T}))$ by the assumption that $\mathrm{max}(f(\mathcal{T}))-\mathrm{min}(f(\mathcal{T}))\geq f(E_++E_-)=1$. We show that replacing $\mathcal{O}(D_{i_0})$ with $\mathcal{O}(D_{i_0}+E_-)$ is possible and will achieve our goal, see $(2)$ of Figure \ref{fig:4}. We have $\mathrm{Ext}^{\mathrm{rk}(N)}(D_{i_0}+E_{L}, D_j)=0$ for all $j\in\{1,\ldots,m\}$ and $L\subsetneqq I_{-}$. Otherwise, we get $D_j-D_{i_0}-E_{L}=\sum_{i\in \{1,\ldots,m\}}(<0)E_i=-E_+-E_-+\sum_{i\in \{1,\ldots,m\}}(\leq0)E_i$ for some $j$ and some $L\subsetneqq I_{-}$. Thus $D_j-D_{i_0}+E_{+}=(E_L-E_-)+\sum_{i\in \{1,\ldots,m\}}(\leq0)E_i$. Then $f(D_j-D_{i_0}+E_{+})=f((E_L-E_-)+\sum_{i\in \{1,\ldots,m\}}(\leq0)E_i)<0$ which contradicts $f(D_{i_0}-E_+)\leq\mathrm{min}(f(\mathcal{T}))$. Then by Lemma \ref{+in}, the line bundle $\mathcal{O}(D_{i_0}+E_-)\in\mathcal{D}(\mathcal{T})$.

\smallskip

Also, we have $\mathrm{Ext}^{\mathrm{rk}(N)}(D_{i_0}+E_{-}, D_j)=0$ for all $j\in\{1,\ldots,m\}$. Otherwise, we get
$D_j-D_{i_0}-E_{-}=\sum_{i\in \{1,\ldots,m\}}(<0)E_i=-E_+-E_-+\sum_{i\in \{1,\ldots,m\}}(\leq0)E_i$ for some $j$. Thus $D_j-D_{i_0}=-E_++\sum_{i\in \{1,\ldots,m\}}(\leq0)E_i$. If the coefficients in $\sum_{i\in \{1,\ldots,m\}}(\leq0)E_i$ are not all zero,
then $f(D_j-D_{i_0}+E_+)=f(\sum_{i\in \{1,\ldots,m\}}(\leq0)E_i)<0$, which contradicts that $f(D_{i_0}-E_+)\leq\widetilde{f}(\mathcal{T})$. %Since $\mathrm{max}(\alpha(\mathcal{T}))-\widetilde{\alpha}(\mathcal{T})<\alpha(E_+)$, so $\alpha(D_{i_0})-\alpha(D_j)\neq \alpha(E_+)$.
%Thus $D_j-D_{i_0}\neq-E_+$ which implies the coefficients in $\sum_{i\in \{1,\ldots,m\}}(\leq0)E_i$ cannot be all zero.
If the coefficients in $\sum_{i\in \{1,\ldots,m\}}(\leq0)E_i$ are all zero, then $D_j-D_{i_0}=-E_+$. This implies $\mathrm{Ext}^{-}(\mathcal{O}(D_{i_0}),\mathcal{O}(D_j))\neq 0$ which contradicts the assumption that $\mathcal{T}$ is a strong exceptional collection.
\begin{figure}[H]
  \includegraphics[width=0.99\textwidth]{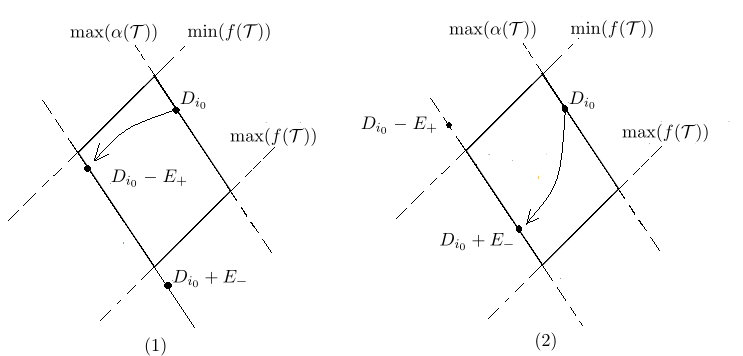}
  \caption{}
  \label{fig:4}
\end{figure}

Then by Lemma \ref{newstrong+}, we get a strong exceptional collection $\mathcal{T}'$ in $\mathcal{D}(\mathcal{T})$ by replacing $D_{i_0}$ with $\mathcal{O}(D_{i_0}+E_-)$ which satisfies $(2)$, $(3)$ and $(4)$ of this Proposition. Since $f(D_{i_0}+E_-)\leq \mathrm{max}({f(\mathcal{T})})$, then $\mathrm{max}({f(\mathcal{T}')})\leq \mathrm{max}({f(\mathcal{T})})$.

\smallskip

 {\bf Case } $f(D_{i_0}+E_-)> \mathrm{max}({f(\mathcal{T})})$. We have $f(D_{i_0}-E_+)> \mathrm{min}({f(\mathcal{T})})$. By the same arguments, we can get a strong exceptional collection $\mathcal{T}'$ in $\mathcal{D}(\mathcal{T})$ by replacing $D_{i_0}$ with $\mathcal{O}(D_{i_0}-E_+)$ which satisfies $(1)$, $(3)$ and $(4)$ of this Proposition, see $(1)$ of Figure \ref{fig:4}. Since $f(D_{i_0}-E_+)> \mathrm{min}({f(\mathcal{T})})$, then $\mathrm{min}({f(\mathcal{T}')})\geq \mathrm{min}({f(\mathcal{T})})$.
\end{proof}

\begin{remark}\label{<=,<}
Let $\mathcal{T}=\{\mathcal{O}(D_1),\ldots,\mathcal{O}(D_n)\}$ be a strong exceptional collection of line bundles with length $n=\mathrm{rk}(K_0(\mathbb{P}_{\mathbf{\Sigma}}))$.
 Assume all line bundles in $\mathcal{T}$ are within a strip of $\alpha$ with width less than $\alpha(E_+)$  and $\mathrm{max}(f(\mathcal{T}))-\mathrm{min}(f(\mathcal{T}))\geq f(E_++E_-)=1$.
After doing the move in Proposition \ref{move1}, we can guarantee that all line bundles in the new strong exceptional collection is within a strip of $\alpha$ with width less or equal to $\alpha(E_+)$. After replacing all $D_j$ in $\mathcal{T}$ such that $\alpha(D_j)=\mathrm{max}(\alpha(\mathcal{T}))$, we get the width of the strip of $\alpha$ to be less than $\alpha(E_+)$.
\end{remark}

\begin{theorem}\label{main2}
Let $\mathbb{P}_{\mathbf{\Sigma}}$ be a Fano toric DM stack with $\mathrm{rank}(\mathrm{Pic} (\mathbb{P}_{\mathbf{\Sigma}}))=2$.
Assume $\mathcal{T}=\{\mathcal{O}(D_1),\ldots,\mathcal{O}(D_n)\}$ be a strong exceptional collection of line bundles with length $n=\mathrm{rk}(K_0(\mathbb{P}_{\mathbf{\Sigma}}))$. Then $\mathcal{T}$ is a full strong exceptional collection.
\end{theorem}
\begin{proof}
Without of loss of generality, we can assume that $\mathrm{max}(\alpha(\mathcal{T}))-\mathrm{min}(\alpha(\mathcal{T}))<\alpha(E_+)=\alpha(-E_-)$ by Proposition \ref{smallwidth}.

\smallskip

Let $D_j$ be an element in $\mathcal{T}$ such that $f(D_j)=\mathrm{min}(f(\mathcal{T}))$. If $\alpha(D_j)=\mathrm{max}(\alpha(\mathcal{T}))$, then by Proposition \ref{move1}, after replacing $\mathcal{O}(D_j)$ with $\mathcal{O}(D_j-E_+)$ or $\mathcal{O}(D_j+E_-)$, we get another strong exceptional collection $\mathcal{T}'$ such that $\sharp(\{D_i\in \mathcal{T}'| f(D_i)=\mathrm{min}(f(\mathcal{T}))\})< \sharp(\mathcal{T}_{\mathrm{min}(f)})$. If $\alpha(D_j)<\mathrm{max}(\alpha(\mathcal{T}))$, then by repeating the process in Proposition \ref{move1} several times, we will get to the situation that $\alpha$ takes maximal value at $D_j$, see Figure \ref{fig:5}.
\begin{figure}[H]
  \includegraphics[width=0.89\textwidth]{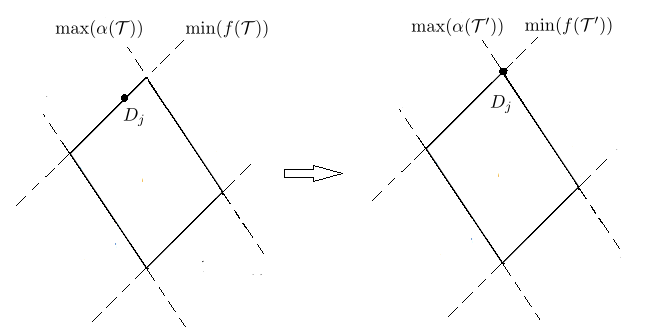}
  \caption{}
  \label{fig:5}
\end{figure}
After replacing all elements in $\mathcal{T}_{\mathrm{min}(f)}$, we get $\mathrm{min}(f(\mathcal{T}))$ increase. Then we continue to apply Proposition \ref{move1}. During the process, we assure that $\mathrm{max}(f(\mathcal{T}))$ does not increase and $\mathrm{min}(f(\mathcal{T}))$ increases. Thus $\mathrm{max}(f(\mathcal{T}))-\mathrm{min}(f(\mathcal{T}))$ decreases. Therefore, we will eventually be in the situation $\mathrm{max}(f(\mathcal{T}))-\mathrm{min}(f(\mathcal{T}))<1$.

\smallskip

Also, by Remark \ref{<=,<}, we get a new strong exceptional collection $\mathcal{S}$ of line bundles in $\mathcal{D}(\mathcal{T})$ such that $\mathrm{max}(\alpha(\mathcal{S}))-\mathrm{min}(\alpha(\mathcal{S}))<\alpha(E_+)$ and $\mathrm{max}(f(\mathcal{S}))-\mathrm{min}(f(\mathcal{S}))<1$. So $\mathcal{S}$ is a full strong exceptional collection by Proposition \ref{p+P}. Thus $\mathcal{D}(\mathcal{T})\supseteq\mathcal{D}(\mathcal{S})=\mathbf{D}^b(coh(\mathbb{P}_{\mathbf{\Sigma}}))$.
\end{proof}

\begin{remark}
When $\mathrm{Pic} (\mathbb{P}_{\mathbf{\Sigma}})$ has torsion, the arguments of this section go through without significant change. The details are left to the reader.
\end{remark}
\section{Comments}\label{comments}
We expect our main result to be valid without the assumption on the rank of Picard group, as stated in Conjecture \ref{anyrank}.
Also, in the case of $\mathrm{rk}(\mathrm{Pic} (\mathbb{P}_{\mathbf{\Sigma}}))=1$, we know that any exceptional collection of line bundles is a strong exceptional collection by Remark \ref{ec,sec}. However, in the case of $\mathrm{rk}(\mathrm{Pic} (\mathbb{P}_{\mathbf{\Sigma}}))=2$,  Theorem \ref{main2} does not tell us that every exceptional collection of maximal length is a full exceptional collection. Thus we hope we can drop strong assumption to ask whether every exceptional collection of maximal length is a full exceptional collection.
The possible future directions include dimension two $\mathrm{rk}(\mathrm{Pic} (\mathbb{P}_{\mathbf{\Sigma}}))=3$ Fano case, and dimension two non-Fano case.
We hope that techniques of this paper can be modified to settle them.
%\begin{conjecture}
%Let $\mathbb{P}_{\mathbf{\Sigma}}$ be a Fano toric DM stack with $\mathrm{rk}(\mathrm{Pic} (\mathbb{P}_{\mathbf{\Sigma}}))=2$.
%Assume $\mathcal{T}=\{\mathcal{O}(D_1),\ldots,\mathcal{O}(D_n)\}$ be an exceptional collection of line bundles with length %$n=\mathrm{rk}(K_0(\mathbb{P}_{\mathbf{\Sigma}}))$. Then $\mathcal{T}$ is a full exceptional collection.
%\end{conjecture}

%We propose the following conjecture in broader sense. While our arguments use the Fano assumption significantly, we don't know if it is really needed.
%\begin{conjecture}
%Any exceptional collection of line bundles of maximal length on a Fano toric DM stack is a full exceptional collection.
%\end{conjecture}

\smallskip

Moreover, in our proofs when we replace $j_0\in\{1,\ldots,n\}$ such that $\alpha(D_{j_0})=\mathrm{max}(\alpha(\mathcal{T}))$
with $\mathcal{O}(D_{i_0}-E_+)$ or $\mathcal{O}(D_{i_0}+E_-)$, the strong exceptional collection "shrinks" in $\mathrm{Pic} (\mathbb{P}_{\mathbf{\Sigma}})$.
We would like to find a more geometric meaning of this phenomenon.


\begin{thebibliography}{999999999}

\bibitem{AO} V. Alexeev, D. Orlov,
{\em Derived categories of Burniat surfaces and exceptional collections}.
Math. Ann. 357 (2013), no. 2, 743-759.

\bibitem{BCS} L. Borisov, L. Chen, G. G. Smith,
{\em The orbifold Chow ring of toric Deligne-Mumforrd stacks}.
J. Amer. Math. Soc. 18(2005), no. 1, 193-215.

\bibitem{BH} L. Borisov, R. Horja,
{\em On the $K$-theory of smooth toric DM stacks}.
Snowbird lectures on string geometry, 21-42, Contemp. Math., 401, Amer. Math. Soc., Providence, RI, 2006.

\bibitem{BHKing} L. Borisov, Z. Hua,
{\em On the conjecture of King for smooth toric Deligne-Mumford stacks}.
Adv. Math. 221(2009), 277-301.

\bibitem{BGrS} Ch. B$\ddot{\mathrm{o}}$hning, H-Ch. Graf von Bothmer, P. Sosna,
{\em On the derived category of the classical Godeaux surface.}
Adv. Math. 243(2013), 203-231.

\bibitem{BGKS} Ch. B$\ddot{\mathrm{o}}$hning, H-Ch. Graf von Bothmer, L. Katzarkov, P. Sosna,
{\em Determinantal Barlow surfaces and phantom categories}.
J. Eur. Math. Soc. (JEMS) 17(2015), no. 7, 1569-1592.

\bibitem{C} D. Cox,
{\em The homogeneous coordinate ring of a toric variety}.
J. Algebraic Geom. 4(1995), no. 1, 17-50.

\bibitem{D} V. I. Danilov,
{\em The geometry of toric varieties}.
Russian Math. Surveys 33 (1978), no. 2, 97-154.

\bibitem{E} A. I. Efimov,
{\em Maximal lengths of exceptional collections of line bundles}.
J. London Math. Soc. (2)90(2014), 350-372.

\bibitem{F} W. Fulton,
{\em Introduction to toric varieties}.
Annals of Mathematics Studies 131, Princeton University Press, Princeton, NJ, 1993.

\bibitem{SOrlov} S. Gorchinskiy, D. Orlov,
{\em Geometric phantom categories}.
Publ. Math. Inst. Hautes ¨¦tudes Sci. 117(2013), 329-349.

\bibitem{GS} S. Galkin, E. Shinder,
{\em Exceptional collections of line bundles on the Beauville surface}.
Adv. Math. 244(2013), 1033-1050.

\bibitem{HP1} L. Hille, M. Perling,
{\em A counterexample to King's conjecture}.
Compos. Math 142(2006), 1507-1521.

\bibitem{HP2} L. Hille, M. Perling,
{\em Exceptional sequences of invertible sheaves on rational surfaces}.
Compos. Math 147(2011), 1230-1280.

\bibitem{I} A. Ishii, K. Ueda,
{\em Dimer models and exceptional collections}.
preprint, 2009, arXiv:0911.4529.

\bibitem{Ka} Y. Kawamata,
{\em Derived categories of toric varieties}.
Michigan Math. J. 54(2006), 517-535.

\bibitem{Ki} A. D. King,
{\em Tilting bundles on some rational surfaces}.
preprint, 1997.

\bibitem{SNo} S. Novakovi$\acute{\mathrm{c}}$,
{\em No phantoms in the derived category of curves over arbitrary fields, and derived characterizations of Brauer-Severi varieties}.
preprint, arXiv:1701.03020

\bibitem{M} M. Michalek,
{\em Family of counterexamples to King's conjecture}.
C. R. Math. Acad. Sci. Pairs 349(2011), 67-69.

\bibitem{N.Prabhu} N. Prabhu-Naik,
{\em Tilting bundles on toric Fano fourfolds}.
J. Algebra 471(2017), 348-398.

\bibitem{PSo} P. Sosna,
{\em Some remarks on phantom categories and motives}.
preprint, arXiv:1511.07711.

\end{thebibliography}
\end{document}